\documentclass[11pt]{article}
\usepackage{amssymb,amsmath}
\usepackage{cases}
\usepackage{amsthm}
\usepackage{amsfonts}
\usepackage{makecell}
\usepackage{cite,color,xcolor}
\usepackage[margin=1.4 in]{geometry}
\usepackage[colorlinks,citecolor=blue,urlcolor=blue]{hyperref}
\usepackage{tikz}

\usetikzlibrary{patterns}
\usepackage[utf8]{inputenc}
\usepackage[pagewise]{lineno}
\usepackage{pgfplots}

\newtheorem{theorem}{Theorem}[section]
\newtheorem{corollary}{Corollary}[section]
\newtheorem{lemma}{Lemma}[section]
\newtheorem{proposition}{Proposition}[section]

\newtheorem{remark}{Remark}[section]

\usepackage{txfonts}

\newcommand{\bal}{\begin{align}}
\newcommand{\bbal}{\begin{align*}}
\newcommand{\beq}{\begin{equation}}
\newcommand{\eeq}{\end{equation}}
\newcommand{\bca}{\begin{cases}}
\newcommand{\eca}{\end{cases}}
\def\div{\mathord{{\rm div}}}
\newcommand{\supp}{{\rm supp}\,}
\newcommand{\pa}{\partial}
\newcommand{\fr}{\frac}
\newcommand{\na}{\nabla}

\newcommand{\cd}{\cdot}
\newcommand{\ep}{\varepsilon}
\newcommand{\dd}{\mathrm{d}}

\newcommand{\R}{\mathbb{R}}

\newcommand{\D}{\mathrm{div}~}

\newcommand{\f}{\left}
\newcommand{\g}{\right}

\begin{document}
\bibliographystyle{plain}
\title{Strong ill-posedness of the 2D Boussinesq equations in supercritical Besov spaces}

\author{
 Jinlu Li
 \footnote{
School of Mathematics and Computer Sciences, Gannan Normal University, Ganzhou 341000, China.
\text{E-mail: lijinlu@gnnu.edu.cn}}
\quad and\quad
Yanghai Yu
\footnote{
School of Mathematics and Statistics, Anhui Normal University, Wuhu 241002, China.
\text{E-mail: yuyanghai214@sina.com} (Corresponding author)}
}
\date{\today}
\maketitle
\begin{abstract}
In this paper, we prove that the 2D Boussinesq equations are strongly ill-posed in the supercritical Besov spaces $B^s_{p,q}$ and Sobolev spaces $W^{s,p}$ with $(p,q)\in(1,\infty)\times [1,\infty]$ and $s\in(0,1+2/p)$ by constructing an initial data with arbitrarily small norm for which the solution of the system exhibits norm inflation almost instantaneously. As a further application, we prove the instability of perturbations near the hydrostatic equilibrium for the 2D Boussinesq equations in the same $B^s_{p,q}$ and $W^{s,p}$.
\end{abstract}



\section{Introduction}

In this paper, we consider the Cauchy problem of the following inviscid Boussinesq equations
\begin{align}\label{2EB}
\begin{cases}
\pa_t u+u\cdot \nabla u+\nabla P=\theta e_2, &\quad (t,\mathbf{x})\in \R^+\times\R^2,\\
\pa_t \theta+u\cdot \nabla \theta=0,&\quad (t,\mathbf{x})\in \R^+\times\R^2,\\
\div\ u=0,&\quad (t,\mathbf{x})\in \R^+\times\R^2,\\
(u,\theta)(0,\mathbf{x})=(u_0,\theta_0)(\mathbf{x}), &\quad \mathbf{x}\in \R^2,
\end{cases}
\end{align}
where $u$ is the velocity satisfying the 2D Euler equations driven by $\theta$, which represents the density or temperature of the fluid, depending on the physical context. Also, $P$ denotes the scalar pressure and $e_2 = (0, 1)$ is the unit vector in the vertical direction. $\theta e_2$ in the velocity equation models the effect of gravity on the fluid motion.

The 2D Boussinesq system \eqref{2EB} is of significant interest in (mathematical) fluid dynamics since the system is widely used to describe the dynamics of the oceans or the atmosphere, where rotation and stratification are important, as well as other physical problems. For more details on other physical problems, see \cite{maj,gil}. Besides physical applications, the 2D Boussinesq system \eqref{2EB} is also known to retain some key features of the 3D Euler and Navier-Stokes equations, such as the vortex stretching mechanism. Indeed, it has been commonly recognized that the growth of the vorticity associated with \eqref{2EB} depends on the temporal accumulation of $\nabla u$, which is a scenario similar to the vortex stretching effect in 3D incompressible flows \cite{MB}. Another important feature is that, the inviscid 2D Boussinesq equations can be identified with the 3D axi-symmetric Euler equations with swirl away from the symmetric axis $r=0$ \cite{MB}.

Before stating our result, we summarize the well-posedness results for the 2D Boussinesq equations, including the cases with the viscosity and the thermal diffusivity:
\begin{align}\label{2EB-0}
\begin{cases}
\partial_t u + u \cdot \nabla u -\nu \Delta u +\nabla P= \theta e_2, &\quad (t,\mathbf{x})\in \R^+\times\R^2,\\
\partial_t \theta + u \cdot \nabla \theta - \kappa \Delta \theta=0, &\quad (t,\mathbf{x})\in \R^+\times\R^2,\\
\div\ u = 0, &\quad (t,\mathbf{x})\in \R^+\times\R^2,\\
(u,\theta)(0,\mathbf{x})=(u_0,\theta_0)(\mathbf{x}), &\quad \mathbf{x}\in \R^2.
\end{cases}
\end{align}

\begin{itemize}
  \item  In the full viscous case $(\nu > 0,\kappa > 0)$, it is known that the 2D Boussinesq system \eqref{2EB-0} is globally well-posed (see \cite{nian1,nian2}).
\end{itemize}
\begin{itemize}
  \item In the partially viscous case, Chae \cite{Chae} proved the global regularity of \eqref{2EB-0} with either the zero diffusivity (\( \nu > 0 \), \( \kappa = 0 \)) or the zero viscosity (\( \nu = 0 \), \( \kappa > 0 \)) in $H^m$ for $m > 2$ (see also \cite{houl}). Hmidi and Keraani \cite{hk} focused on the case \( \nu = 0 \) and \( \kappa > 0 \), and proved the global existence and uniqueness of \eqref{2EB-0} for initial data $u_0\in B_{p,1}^{1+2/p}$ and $\theta_0\in B_{p,1}^{-1+2/p}\cap L^q$ with $q>2$.
\end{itemize}
\begin{itemize}
  \item In the inviscid case (\(\nu = \kappa = 0\)), \eqref{2EB-0} reduces to \eqref{2EB}. There have been extensive studies but only involving the local-in-time theory for \eqref{2EB}. Chae and Nam \cite{chaeN} proved a local well-posedness (existence and uniqueness) result and also a blow-up criterion for the system \eqref{2EB} in Sobolev spaces $W^{s,2}(\R^2)$ with $s > 2$. In the context of Besov spaces, Liu, Wang and Zhang \cite{lwz} established the local well-posedness
of solutions to \eqref{2EB} and a blow-up criterion in critical space \( B_{p,1}^{1+2/p}(\mathbb{R}^2) \) with $ 1 <p <\infty$. Yuan \cite{yb} proved a result of local existence and uniqueness of solutions to \eqref{2EB} in $B^s_{p,q}(\R^2)$ with $s\geq 1+2/p, 1 < p < \infty$ and $1 \leq q \leq \infty$, where $q= 1$ if $s= 1+2/p$, as well as a blow-up criterion.
\end{itemize}

However, we notice that, there have few local well-posedness result for the 2D inviscid Boussinesq equations \eqref{2EB} in the supercritical Besov spaces $B^s_{p,q}(\R^2)$ with $s<1+2/p$. Inspired by the recent work by Luo \cite{luo1d} who proved that the 3D Euler equations are strongly ill-posed in supercritical Sobolev spaces $H^s$ for any
$0 < s <5/2$ by constructing an initial velocity field with arbitrarily small $H^s$ norm for which the unique local-in-time smooth solution of the 3D Euler equation develops large $\dot{H}^s$ norm inflation almost instantaneously (see \cite{luo2} for 3D Navier-Stokes equations), we will prove that the 2D Boussinesq system \eqref{2EB} is strong ill-posed in the supercritical Besov spaces $B^s_{p,q}(\R^2)$.
Our main result is the following.

\begin{theorem}\label{th1}
Assume that $p\in(1,\infty)$ and $s\in(0,1+\fr2p)$. For any $\ep> 0$, there exists a solution $(u(t),\theta(t))\in \mathcal{C}([0,T_0]; B^\infty_{p,1})$ of the Cauchy problem \eqref{2EB} and $0< T_0 < \ep$ such that $(u_0,\theta_0) \in {\cal S}(\R^2)$ satisfying
$$
\|u_0,\theta_0\|_{B^s_{p,1}}\leq \ep \quad\text{but}
\quad \|\theta(T_0)\|_{B^s_{p,\infty}} > \frac{1}{\ep}.$$
\end{theorem}
Notice that the fact $B^s_{p,1} \hookrightarrow B^s_{p,q} ( W^{s,p} )\hookrightarrow B^s_{p,\infty}$ with $s>0$ and $(p,q)\in(1,\infty)\times [1,\infty]$, we have
\begin{corollary}\label{cor1}
The Boussinesq system \eqref{2EB} is strong ill-posed in both $B^s_{p,q}(\R^2)$ and $W^{s,p}(\R^2)$ for any $s\in(0,1+2/p)$ and $(p,q)\in(1,\infty)\times [1,\infty]$ in the sense that the data-to-solution map is strong instability.
\end{corollary}
\begin{remark}
We should mention that, our ill-posedness result holds within the framework of Besov spaces which covers the Sobolev spaces $H^s(\R^2)$ with $s\in(0,2)$ since $B^s_{2,2}=H^s$.
Apart from the result in $W^{1,p}(\mathbb{T}^3)$ of DiPerna-Lions \cite{DM}, very little was known on the classical Euler equation in the supercritical $W^{s,p}$ with $s<1+d/p$, while Corollary \ref{cor1} is somewhat surprising. In particular, our result shows the failure of the boundedness of the data-to-solution map in $W^{1,p}(\R^2)$ for any $1<p<\infty$. More precisely, there is no continuous function $\tau\mapsto f(\tau)$ such that any smooth solution of the Cauchy problem \eqref{2EB} satisfies the estimate
\begin{align*}
\|\theta(\cdot,t)\|_{W^{1,p}}\leq f(\|u_0,\theta_0\|_{W^{1,p}}).
\end{align*}
Our new observation is that, the instability mechanism is driven by the $\theta$-equation, which is not possible for the Euler equation since the conservation of $L^p$-norm of voticity, i.e. $\|\omega\|_{L^p}=\|\omega_0\|_{L^p}$ for $p\in[1,\infty]$.
\end{remark}

The hydrostatic balance plays many important roles in geophysics and astrophysics \cite{ped,wit}. Understanding the stability of perturbations near the hydrostatic balance is important both mathematically and physically.
In this paper, we take the hydrostatic equilibrium
\begin{align}\label{pht}
u_{\text{eq}} = \mathbf{0}, \quad \theta_{\text{eq}} = \beta y, \quad P_{\text{eq}} = \frac{1}{2} \beta y^2,
\end{align}
where $\beta\neq0$ and introduce the perturbation
\[\tilde{u}=u-u_{\text{eq}}, \quad \tilde{\theta}=\theta - \theta_{\text{eq}}, \quad \tilde{P} = P - P_{\text{eq}},\]
We consider the time evolution of the perturbations around a stable state in hydrostatic balance, and then \((\tilde{u}, \tilde{\theta}, \tilde{P})\) should satisfy
\begin{align*}
\begin{cases}
\partial_t \tilde{u} + \tilde{u}\cdot \nabla \tilde{u} +\nabla \tilde{P}= \tilde{\theta} e_2, \\
\partial_t \tilde{\theta} + \tilde{u} \cdot \nabla\tilde{\theta} = -\beta \tilde{u}_2, \\
\div\ \tilde{u}= 0, \\
\tilde{u}(0, \mathbf{x}) =u_0(\mathbf{x}), \quad \tilde{\theta}(0, \mathbf{x}) =\theta_0(\mathbf{x}) - \beta y.
\end{cases}
\end{align*}
\begin{corollary}\label{cor2}
The 2D Boussinesq equations around the hydrostatic equilibrium \eqref{pht} for $\beta\neq0$ are strongly ill-posed in both $B^s_{p,q}(\R^2)$ and $W^{s,p}(\R^2)$ for any $s\in(0,1+2/p)$ and $(p,q)\in(1,\infty)\times [1,\infty]$.
\end{corollary}
\begin{remark}
Bianchini et.al. \cite{BHI} considered the perturbations of the simplest stably stratified density profile, namely \({\theta}_{\text{eq}}(y) = - y\), and proved that the 2D Boussinesq equations around the stably stratified steady state in vorticity form are strongly ill-posed in $L^{\infty}(\R^2)$ (see \cite{E-ARMA} for mild ill-posedness result). Compared with the ill-posedness result in $W^{1,\infty}(\R^2)$, Corollary \ref{cor2} holds in $W^{1,p}(\R^2)$ for any $1<p<\infty$.
\end{remark}
{\bf Idea of the proof.}\,Let us give an overview of the main ideas leading to our results.

First, we construct the approximate system by neglecting the linear term of the original system
\begin{align}\label{app0}
\begin{cases}
\pa_t U+U\cdot \nabla U+\nabla \bar{P}=0, \\
\pa_t \Theta+U\cdot \nabla \Theta=0,\\
\div\ U=0,\\
(U,\Theta)(0,\mathbf{x})=(u_0,\theta_0)(\mathbf{x}).
\end{cases}
\end{align}
Choosing special initial data $(u_0,\theta_0)$ which inspired by \cite{luo1d} for the Euler equations, we find a global solution $(U,\Theta)$ of \eqref{app0} which we call it the good approximate solution for \eqref{2EB}.

Second, we show that the approximate solution  $(U,\Theta)$ exhibit norm inflation instantaneously in timescale $T\approx\frac{\ep^{-\alpha}}{\|u_0,\theta_0\|_{W^{1,\infty}}}$ for small $\ep>0$ and some $\alpha>0$.

Last, we perform perturbation analysis to show that the difference $(u-U,\theta-\Theta)$ between the real solution and the approximation solution is controlled up to the norm inflation time.

\section{Preliminaries}\label{sec2}
\subsection{Notations}
For the sake of simplicity, we set $r = |\mathbf{x}|=\sqrt{x^2 + y^2}$ with $\mathbf{x}=(x,y)\in\R^2$. We will also use the polar coordinate $\mathbf{x}=(r\cos\alpha,r\sin\alpha)$ with $\alpha=\arctan\frac{y}{x}$ and denote the polar coordinate unit basis vectors
$\vec{e}_r = (\cos \alpha, \sin \alpha), \ \vec{e}_\alpha = (-\sin \alpha, \cos \alpha).$
The vector or metric $\nabla f$ denotes the gradient of $f$ with respect to the space variable, whose $(i,j)$-th component is given by $(\nabla u)_{ij}=\pa_iu_j$ with $1\leq i,j\leq 2$. Throughout the paper, for \( k \in \mathbb{N} \), \( \nabla^k \) refers to the full gradient in \( \mathbb{R}^2 \). For a vector- or tensor-valued function \( f \), its modulus \( |f| \) denotes the square root of the sum of squares of each component.
Throughout this paper, $C$ stands for some positive constant independent of $n$, which may vary from line to line.
The symbol $A\approx B$ means that $C^{-1}B\leq A\leq CB$.
Given a Banach space $X$, we denote its norm by $\|\cdot\|_{X}$. We also use the simplified notation $\|f_1,\cdots,f_n\|_{X}:=\|f_1\|_{X}+\cdots+\|f_n\|_{X}$.
For $I\subset\R$, we denote by $\mathcal{C}(I;X)$ the set of continuous functions on $I$ with values in $X$.
We use $\mathcal{S}(\R^2)$ and $\mathcal{S}'(\R^2)$ to denote Schwartz functions and the tempered distributions spaces on $\R^2$, respectively.
\subsection{Functional Spaces}
Next, we will recall some facts about the Littlewood-Paley decomposition and the nonhomogeneous Besov spaces (see \cite{BCD} for more details).
Choose a radial, non-negative, smooth function $\vartheta:\R^d\mapsto [0,1]$ such that ${\rm{supp}} \;\vartheta\subset B(0, 4/3)$ and $\vartheta(\xi)\equiv1$ for $|\xi|\leq3/4$.
Setting $\varphi(\xi):=\vartheta(\xi/2)-\vartheta(\xi)$, then we deduce that $\varphi$ satisfies
${\rm{supp}} \ \varphi\subset \left\{\xi\in \R^d: 3/4\leq|\xi|\leq8/3\right\}$ and $\varphi(\xi)\equiv 1$ for $4/3\leq |\xi|\leq 3/2$.
The nonhomogeneous dyadic blocks are defined as follows
\begin{align*}
\forall\, u\in \mathcal{S'}(\R^d),\quad \Delta_ju=0,\; \text{if}\; j\leq-2;\quad
\Delta_{-1}u=\vartheta(D)u;\quad
\Delta_ju=\varphi(2^{-j}D)u,\; \; \text{if}\;j\geq0.
\end{align*}

{\bf Besov spaces.}\,
Let $s\in\mathbb{R}$ and $(p,r)\in[1, \infty]^2$. The nonhomogeneous and homogeneous Besov spaces are defined, respectively,
$$
B^{s}_{p,q}:=\f\{f\in \mathcal{S}':\;\|f\|_{B^{s}_{p,q}}:=\left\|2^{js}\|\Delta_jf\|_{L_x^p}\right\|_{\ell^q(j\geq-1)}<\infty\g\}
$$
and
$$
\dot{B}^{s}_{p,q}:=\f\{f\in \mathcal{S}'_h:\;\|f\|_{\dot{B}^{s}_{p,q}}:=\left\|2^{js}\|\dot{\Delta}_jf\|_{L_x^p}\right\|_{\ell^q(j\in \mathbb{Z})}<\infty\g\},
$$
where \( \ell^\infty \) is used for \( q = \infty \), i.e., the supremum over \( q \).

{\bf Sobolev spaces.}\,
We recall the definition of Sobolev spaces. For any integer \( k \in \mathbb{N} \) and $1 \le p \le \infty$:
\begin{equation*}
\|f\|_{W^{k,p}} =\|f\|_{L^{p}}+ \|f\|_{\dot{W}^{k,p}}\quad\text{with}\quad \|f\|_{\dot{W}^{k,p}} = \|\nabla^k f\|_{L^p}.
\end{equation*}
For any real \( s \in \mathbb{R} \) and \( 1 < p < \infty \):
\begin{equation*}
\|f\|_{W^{s,p}} = \|(\mathrm{Id}-\Delta)^{s/2} f\|_{L^p}\quad\text{with}\quad
 \|f\|_{\dot{W}^{s,p}} = \|(-\Delta)^{s/2} f\|_{L^p}.
\end{equation*}
It is well known that, the two definitions coincide when \( 1 < p < \infty \). We emphasize that when \( p = \infty \), we only use the definition of integer Sobolev spaces.
When \( p = 2 \), we denote \( H^s := W^{s,2} \) and $\dot{H}^s:=\dot{W}^{s,2}$. For any $s>0$ and $(p,r)\in[1, \infty]^2$, then $B^{s}_{p,r}=\dot{B}^{s}_{p,r}\cap L^p$ and
$\|f\|_{B^{s}_{p,r}}\approx \|f\|_{L^{p}}+\|f\|_{\dot{B}^{s}_{p,r}}.$

\section{Proof of Theorem \ref{th1}}\label{sec1}

In this section, we construct the approximate solution that exhibits norm inflation in $B^s_{p,q}$ and then prove Theorem \ref{th1}. We divide the proof into six steps.

{\bf Choice of Parameters}. From now on we let $s\in(0,1+2/p)$ and $p\in(1,\infty)$. Throughout this paper, the parameters are chosen as follows:
\begin{align*}
&n\in \f\{2^j,j\in \mathbb{N}\g\}, \quad n\gg 1,\\
&\varepsilon_n=\f(\frac{1}{\log\log n}\g)^{\fr{sp}{8(1+p)}}, \\
&\delta_n=\varepsilon_nn^{\fr{2}{p}-s},\\
&\Gamma_n=\frac{1}{\log n},\\
&T_n=n^{s-1-\fr{2}{p}}\ep_n^{-(1+\fr2s)}.
\end{align*}
With the above, it holds that the simple facts which will be used in the sequel
$$e^{C_1\varepsilon_n^{-\frac{4}{s}(1+\fr1p)}}=e^{C_1\sqrt{\log\log n}}\leq C_2\log n
$$
for some universal positive constants $C_1$ and $C_2$. Also,
$$n\delta_nT_n=\ep_n^{-\fr2s}\gg1\quad\Leftrightarrow\quad (n\delta_nT_n)^s=\ep_n^{-2}.$$
\subsection{Construction of initial data}\label{subsec1}
Let $f(r)$ and $h(r)$ be radial functions satisfying $f, h\in C_c^\infty (\R)$ and
\bbal
&\mathrm{supp} \ h\subset \f\{1 \leq r \leq \frac32\g\}, \quad \mathrm{supp} \ f\subset \f\{\frac12 \leq r \leq 2\g\}\\
&h(r)=1 \quad \mathrm{for} \quad  \f\{\frac98 \leq r \leq \frac{11}{8}\g\},
\\
&f'(r)=1 \quad \mathrm{for} \quad  \f\{1 \leq r \leq \frac32\g\}.
\end{align*}
We define the initial data $u_0=(u_{0,1},u_{0,2})$ and $\theta_0$ as follows:
\bbal
&u_{0,1}=-\delta_n\pa_y[f(nr)]=-\delta_nf'(nr)\frac{y}{r},
\\&u_{0,2}=\delta_n\pa_x[f(nr)]=\delta_nf'(nr)\frac{x}{r},
\\&\theta_0=\delta_nh(nr)\sin\f(\arctan\frac{y}{x}\g)=\delta_nh(nr)\frac{y}{r}.
\end{align*}
Obviously, $\div\ u_0=0$.

\subsection{Estimation of initial data}\label{subsec2}
 To estimate the $B^s_{p,q}$-norm of $(u_{0},\theta_0)$, we need to establish the following Lemma.
\begin{lemma}\label{lem0}
Let $f\in C^\infty_0(\R)$ be given above, then for $g=f'(nr)\frac{y}{r}$ or $g=f'(nr)\frac{x}{r}$, we have
\bbal
\|g\|_{B^s_{p,q}}\approx n^{s-\fr{2}{p}}.
\end{align*}
\end{lemma}
\begin{proof}
For simplicity, we prove the case when $g=f'(nr)\frac{y}{r}$. We define
\[
G(\mathbf{x}):=f'(r)\frac{y}{r},\qquad \mathbf{x}=(x, y)\in\mathbb R^2.
\]
Since $f\in C_c^\infty(\fr12,2)$, we have $G\in C_c^\infty(\mathbb R^2)$ and
$g(\mathbf{x})=G(n \mathbf{x}).$
By the scaling properties of homogeneous Besov spaces, we have
\bbal
\|g\|_{\dot{B}^s_{p,q}}=\|G(n\cdot)\|_{\dot{B}^s_{p,q}}=n^{s-\fr{2}{p}}\|G\|_{\dot{B}^s_{p,q}}.
\end{align*}
Similarly we have
\[
\|g\|_{L^p}=n^{-\fr2p}\|G\|_{L^p}\approx n^{-\fr2p},
\]
where we have used the fact $\|G\|_{L^p}\approx 1$. In fact,
it is straightforward to compute
\bbal
\|G\|^p_{L^p}=&\int_0^{2\pi}\int_0^\infty |f'(r)|^p|\sin\eta|^pr\dd r\dd\eta
=\int_0^{2\pi}|\sin\eta|^p\dd\eta\int_0^\infty |f'(r)|^pr\dd r.
\end{align*}
Thus we complete the proof.
\end{proof}
As a direct consequence, we have
\begin{proposition}\label{pro0}
Let $s>0$ and $(p,q)\in(1,\infty)\times [1,\infty]$, we have the estimates
\bbal
&\|u_{0},\theta_0\|_{W^{1,\infty}}\approx \varepsilon_nn^{1+\fr{2}{p}-s},\quad
\|u_{0},\theta_0\|_{B^s_{p,q}}\approx \varepsilon_n.
\end{align*}
\begin{remark} We would like to mention that, our construction $u_0$ will be non-Lipchitz since the $W^{1,\infty}$-norm of $u_0$ will go to infinity as $n\to\infty$, which leads to the loss of regularity of solutions and thus our norm inflation of solutions.
\end{remark}
\end{proposition}

\subsection{Construction of Approximation solution}\label{subsec3}

Motivated by \cite{luo1d}, we define the approximate solution
\bal
&U(t,\mathbf{x})=u_0(\mathbf{x})=(u_{0,1},u_{0,2})(\mathbf{x})=u_{0,\alpha}(r)\vec{e}_{\alpha}, \quad u_{0,\alpha}(r):=\delta_nf'(nr),\label{sol1}\\
&\Theta(t,\mathbf{x})=\delta_nh(nr)\sin\f(\arctan\frac{y}{x}-\frac{\delta_nt}{r}\g).\label{sol2}
\end{align}
We observe that, $U$ is a stationary solution to the Euler system in $\R^2$, i.e.,
\begin{align}\label{seu}
\begin{cases}
\pa_tU+U\cdot \nabla U+\nabla P_0=0, \\
P_0=\int^r_{\frac12}\frac{u^2_{0,\alpha}(r')}{r'}\dd r', \\
\div\ U=0,
\end{cases}
\end{align}
and $\Theta$ solves the transport equation
\begin{align}\label{trans}
\begin{cases}
\pa_t \Theta+U\cdot \nabla \Theta=0, \\
\Theta(0,\mathbf{x})=\theta_0(\mathbf{x}).
\end{cases}
\end{align}
In fact, from \eqref{sol1}, we have
\begin{align*}
U\cdot\nabla U=\frac{u_{0,\alpha}(r)}{r}\pa_{\alpha}(u_{0,\alpha}(r)\vec{e}_{\alpha})
=-\frac{u^2_{0,\alpha}(r)}{r}\vec{e}_{r}=-\na\f(\int^r_{\frac12}\frac{u^2_{0,\alpha}(r')}{r'}\dd r'\g).
\end{align*}
Hence, we have $\pa_tU+U\cd\na U+\na P_0=0$ and $(\mathrm{Id}-\na\Delta^{-1}\D)(U\cd\na U)=0$.

We write the transport equation \eqref{trans} in polar coordinate as
\begin{align*}
\begin{cases}
\partial_t \Theta+\frac{\delta_nf'(nr)}{r}\partial_\alpha \Theta=0, \\
\Theta_0=\delta_nh(nr)\sin \alpha,
\end{cases}
\end{align*}
which has the solution
\begin{equation*}
\Theta(t,\mathbf{x})=\delta_nh(nr)\sin\left(\arctan\frac{y}{x}-\frac{\delta_nf'(nr)}{r}t\right)=\delta_nh(nr)\sin\left(\arctan\frac{y}{x}-\frac{\delta_n}{r}t\right),
\end{equation*}
here we notice that $f'=1$ on $\supp h$.

\subsection{Norm inflation of Approximation solution}\label{subsec4}

\begin{proposition}[Estimation of Approximate solution for $(U,\Theta)$]\label{pro1}
For any $t>0$ and $k\geq s$, there hold that
\bbal
&\|U\|_{L^p}\leq C\varepsilon_nn^{-s}, \quad\|\Theta(t)\|_{L^p}\leq C\varepsilon_nn^{-s},\quad \|U\|_{B^{k}_{p,q}}\approx\ep_nn^{k-s}.
\end{align*}
\end{proposition}
\begin{proof}
It follows from Lemma \ref{lem0} that the desired results.
\end{proof}
To estimate the $B^s_{p,q}$-norm of $\Theta$, we need to establish the following crucial Lemma.
\begin{lemma}\label{lem2}
 Assume that $s\in(0,1+2/p)$ with $p\in(1,\infty)$. Let $h\in C^\infty_0(\R)$ be given above. Define \[
g(t,\mathbf{x})=h(n r)\sin\f(\arctan\frac{y}{x}-\frac{\delta_nt}{r}\g).
\]
Then we have
\bbal
\|g\|_{B^s_{p,q}}\leq C n^{s-\fr{2}{p}}\cdot \max\f\{1,n\delta_nt\g\}^s.
\end{align*}
If $n\delta_nt \gg 1$, we also have
\bbal
\|g\|_{B^s_{p,q}}\approx  n^{s-\fr{2}{p}}\cdot(n\delta_nt)^s.
\end{align*}
\end{lemma}
\begin{proof} Obviously, we have
\bbal
\|g\|_{L^p}\leq Cn^{-\fr{2}{p}}.
\end{align*}
For short,
\[
\phi(\mathbf{x}):=h(n r),\qquad\varphi(\mathbf{x}):=\arctan\frac{y}{x}-\frac{\delta_nt}{r}.
\]
Note the derivative on $f$ gives a factor $n$, thus, for all integers $k\ge1$
\[
|\nabla^k\phi|\leq |h^{(k)}(n r)| n^{k}.
\]
For $k\geq1$, we have
\bbal
\nabla^kg&=h(n r) (\nabla\varphi)^k\sin^{(k)}\varphi+nh^{(k)}(n r)\frac{\mathbf{x}}{r}\sin\varphi+\cdots.
\end{align*}
We observe that the first term $h(n r) (\nabla\varphi)^k\sin^{(k)}\varphi$ is the main contribution. To see this, we take $k=1$ for example.
Due to
$$\nabla \varphi=\frac{1}{r^2}(-y,x) + \frac{\delta_nt}{r^3}(x,y),
$$
we have
\bbal
\nabla g&= nh'(n r)\frac{\mathbf{x}}{r}\sin\varphi+h(n r)\cos\varphi \f(\frac{\mathbf{x}^{\bot}}{r^2} +\delta_nt \frac{\mathbf{x}}{r^3}\g).
\end{align*}
On $\supp h(n r)$, which implies that $r\approx n^{-1}$, one has
\[
|\nabla\varphi|\leq \frac{1}{r} \f(1+\frac{\delta_nt}{r}\g)\leq 2n\cdot\max\{1,n\delta_nt\}.
\]
If $n\delta_nt \gg 1$, we furthermore deduce that
\[
|\nabla\varphi|\geq n\cdot (n\delta_nt-1)\geq \fr12 n\cdot (n\delta_nt).
\]
Noticing that
$$\|nh'(n r)\frac{\mathbf{x}}{r}\sin\varphi\|_{L^p}\leq Cn^{1-\frac2p},$$
we have
\bbal
\|g\|_{W^{1,p}}\leq Cn^{1-\fr2p}\cdot \max\{1,n\delta_nt\},
\end{align*}
and in particular for $n\delta_nt \gg 1$
\bbal
\|g\|_{W^{1,p}}\geq  cn^{1-\fr2p}\cdot (n\delta_nt),
\end{align*}
where we have used
\bbal
\f\|h(nr)\cos\f(\arctan\frac{y}{x}-\frac{\delta_nt}{r}\g)\g\|^p_{L^p}
=&\int_0^{2\pi}|\cos\eta|^p\dd\eta\int_0^\infty |h(nr)|^pr\dd r\approx n^{-2}.
\end{align*}
For the case $k\geq2$, similarly, we have
\bbal
\|g\|_{W^{k,p}}\leq Cn^{k-\fr2p}\cdot \max\{1,n\delta_nt\}^k,
\end{align*}
and in particular for $n\delta_nt \gg 1$
\bal\label{hh}
\|g\|_{W^{k,p}}\approx n^{k-\fr2p}\cdot (n\delta_nt)^k.
\end{align}
Thus, by the interpolation, we obtain the upper bound
\bbal
\|g\|_{B^s_{p,q}}\leq \|g\|^\alpha_{L^p}\|g\|^{1-\alpha}_{W^{3,p}}\leq Cn^{s-\fr{2}{p}}\cdot \max\f\{1,n\delta_nt\g\}^s.
\end{align*}
To obtain the lower  bound, we divide two cases.
For $s\in(0,1)$, we have
\bbal
\|g\|_{W^{1,p}}\leq \|g\|^\alpha_{B^s_{p,\infty}}\|g\|^{1-\alpha}_{W^{2,p}},
\end{align*}
which follows from \eqref{hh} that
\bbal
\|g\|_{B^s_{p,\infty}}\geq cn^{s-\fr{2}{p}}\cdot (n\delta_nt)^s.
\end{align*}
For $s\in[1,1+2/p)$, we have
\bbal
\|g\|_{W^{2,p}}\leq \|g\|^\alpha_{B^s_{p,\infty}}\|g\|^{1-\alpha}_{W^{3,p}},
\end{align*}
which follows from \eqref{hh} that
\bbal
\|g\|_{B^s_{p,\infty}}\geq cn^{s-\fr{2}{p}}\cdot (n\delta_nt)^s.
\end{align*}
This ends the proof of Lemma \ref{lem2}.
\end{proof}
\begin{remark}\label{re1}
From the above proof, we deduce that Lemma \ref{lem2} holds for any $s>0$.
\end{remark}
From now on, we set $T_n:= n^{s-1-\fr{2}{p}}\ep_n^{-(1+\fr2{s})}$. We should notice that, the choice of $T_n$ is to ensure the norm inflation of approximate solution.

\begin{proposition}[Norm inflation of Approximate solution for $\Theta$]\label{pro2}
Let $\gamma\in(0,1+\fr2{p}]$, the function $\Theta$ given by
\eqref{sol2} satisfies
\begin{align*}
\|\Theta(T_n)\|_{B^{\gamma}_{p,q}}\approx \ep_n^{1-\frac{2\gamma}{s}}n^{\gamma-s}.
\end{align*}
In particular,
$$\|\Theta(T_n)\|_{B^{s}_{p,q}}\approx\ep^{-1}_n.$$
\end{proposition}
\begin{proof}
Notice that $n\delta_nT_n= \ep^{-\fr2{s}}_n\gg 1$, by Lemma \ref{lem2}, it is trivial that
\bbal
\|\Theta(T_n)\|_{B^{\gamma}_{p,r}}&= \delta_n\f\|h(nr)\sin\f(\arctan\frac{y}{x}-\frac{\delta_n}{r} T_n\g)\g\|_{B^{\gamma}_{p,q}}\\
&\approx\delta_n n^{\gamma-\fr{2}{p}}(n\delta_nT_n)^\gamma\approx \ep_n^{1-\frac{2\gamma}{s}}n^{\gamma-s}.
\end{align*}
Then we complete the proof of Proposition \ref{pro2}.
\end{proof}

\subsection{Error equation and estimation of error}\label{subsec5}
Define the difference
$$v:=u-U,\;\;\beta:=\theta-\Theta,$$
then we have
\begin{align}\label{er}
\begin{cases}
\pa_t v+u\cdot \nabla v=\mathbf{P}\f(\beta e_2+\Theta e_2-v\cdot\na U\g)+\mathbf{Q}(v\cdot \nabla u), \\
\pa_t \beta+u\cdot \nabla \beta+v\cd\na \Theta=0,\\
(v,\beta)(0,\mathbf{x})=(0,0),
\end{cases}
\end{align}
where we have used $\mathbf{P}=\mathrm{Id}-\na\Delta^{-1}$ and $\mathbf{Q}=\na\Delta^{-1}\D$.
\begin{proposition}[Smallness of Error]\label{pro3}
    For $0 \le t \le T_n$, there holds
\begin{align}\label{ok0}
\begin{cases}
\|(v,\beta)(t)\|_{L^p}\leq \Gamma^3_nn^{-s}, \\
\|(v,\beta)(t)\|_{B^{2/p}_{p,1}}\leq \Gamma^{3}_n n^{\fr{2}{p}-s},\\
\|(v,\beta)(t)\|_{B^{1+2/p}_{p,1}}\leq \Gamma^{\fr32}_n n^{1+\fr{2}{p}-s}.
\end{cases}
\end{align}
In particular, by the interpolation, we have $$\|v,\beta\|_{L^\infty_tB^{s}_{p,1}}\leq \Gamma_n.$$
\end{proposition}
\begin{proof}
Since $(v,\beta)(0)=0$, we can assume that in a short time $0\le t\le t_0$, it holds the following \textbf{Bootstrap assumption:}
\begin{align*}
\|(v,\beta)(t)\|_{L^p}\leq \Gamma^3_nn^{-s}, \quad \|(v,\beta)(t)\|_{B^{2/p}_{p,1}}\leq \Gamma^{3}_n n^{\fr{2}{p}-s},
\quad\|(v,\beta)(t)\|_{B^{1+2/p}_{p,1}}\leq \Gamma_n n^{1+\fr{2}{p}-s}, \quad t \in[0,t_{0}].
\end{align*}
Let us define the maximum time of the bootstrap bound as
\bbal
T_*:=\sup\f\{t\in[0,T_n]: \|(v,\beta)(t)\|_{L^p}\leq \Gamma^3_nn^{-s}, \
 \|(v,\beta)(t)\|_{B^{\fr{2}{p}}_{p,1}}\leq \Gamma^{3}_n n^{\fr{2}{p}-s},\ \|(v,\beta)(t)\|_{B^{1+\fr{2}{p}}_{p,1}}\leq \Gamma_n n^{1+\fr{2}{p}-s}\g\}.
\end{align*}
By continuity, $T_{*}>0 $. It suffices to prove $T_*=T_n$. Assume that $T_*<T_n$, our aim is to prove
\begin{align}\label{ok}
\begin{cases}
\|(v,\beta)(T_*)\|_{L^p}\leq \Gamma^4_nn^{-s}, \\
\|(v,\beta)(T_*)\|_{B^{2/p}_{p,1}}\leq \Gamma^{4}_n n^{\fr{2}{p}-s},\\
\|(v,\beta)(T_*)\|_{B^{1+2/p}_{p,1}}\leq \Gamma^{\fr32}_n n^{1+\fr{2}{p}-s},
\end{cases}
\end{align}
which extends the bootstrap bound slightly
beyond $T_*$ by continuity argument and leads to a contradiction with the definition of $T_*$. Thus $T_*=T_n$ and
\eqref{ok0} follows on $[0,T_n]$.

By the definition of $T_*$, and using Propositions \ref{pro1} and \ref{pro2}, we have for $t \in[0,T_*]$
\begin{equation*}
     \|(v,\beta)(t)\|_{B^{1+2/p}_{p,1}}+\|(U,\Theta)(t)\|_{B^{1+2/p}_{p,1}} \lesssim \ep_n^{1-\frac{2}{s}(1+\fr2p)} n^{1+\fr{2}{p}-s},
\end{equation*}
which implies that for $t \in[0,T_*]$
\begin{align}\label{eq-a}
     \int_0^t\|(v,\beta)(\tau)\|_{B^{1+2/p}_{p,1}} +\|(U,\Theta)(\tau)\|_{B^{1+2/p}_{p,1}}\dd \tau  &\leq CT_n\ep^{1-\frac{2}{s}(1+\fr2p)}_n n^{1+\fr{2}{p}-s}\leq C\ep_n^{-\frac{4}{s}(1+\fr1p)}.
\end{align}

{\bf Step 1: $L^{p}$ estimate.}\, A standard $L^p$ $(1<p<\infty)$ estimate gives that for $t \in[0,T_*]$
\bbal
\|(v,\beta)(t)\|_{L^p}&\leq C\int^t_0\|(v,\beta)(\tau)\|_{L^p}\f(1+\|u,U,\Theta\|_{W^{1,\infty}}\g)\dd \tau+ C\int^t_0\|\Theta\|_{L^p}\dd \tau
\\&\leq C\int^t_0\|(v,\beta)(\tau)\|_{L^p}\f(1+\|u,U,\Theta\|_{B^{1+2/p}_{p,1}}\g)\dd \tau+ Ct\ep_nn^{-s}.
\end{align*}
Using Gronwall's inequality and \eqref{eq-a} yield for $t \in[0,T_*]$
\bbal
\|(v,\beta)(t)\|_{L^p}&\leq Ct\ep_nn^{-s}\exp\f\{C\ep_n^{-\frac{4}{s}(1+\fr1p)}\g\}
\\&\leq C \log\log n\cdot n^{s-1-\fr{2}{p}}\cdot\log n\cdot n^{-s}\\
&\leq \Gamma^4_nn^{-s}.
\end{align*}

{\bf Step 2: $B^{2/p}_{p,1}$ estimate.}\, Using the regularity theory of transport equation (see \cite[Theorem 3.14]{BCD}), we have
\bbal
\|(v,\beta)(t)\|_{B^{2/p}_{p,1}}&\leq C\int^t_0\|(v,\beta)(\tau)\|_{B^{2/p}_{p,1}}(1+\|u\|_{B^{1+2/p}_{p,1}}+\|U,\Theta\|_{W^{1,\infty}})\dd \tau
\\& \quad + C\int^t_0\|(v,\beta)(\tau)\|_{L^\infty}\|U,\Theta\|_{B^{1+2/p}_{p,1}}\dd \tau+C\int^t_0\|\Theta\|_{B^{2/p}_{p,1}}\dd \tau\\
&\leq C\int^t_0\|(v,\beta)(\tau)\|_{B^{2/p}_{p,1}}\f(1+\|v,\beta\|_{B^{1+2/p}_{p,1}}+\|U,\Theta\|_{B^{1+2/p}_{p,1}}\g)\dd \tau+ Ct\ep_n^{1-\frac{4}{sp}},
\end{align*}
which implies
\bbal
\|(v,\beta)(t)\|_{B^{2/p}_{p,1}}&\leq Ct\ep_n^{1-\frac{4}{sp}}n^{\fr{2}{p}-s}\exp\f\{C\ep_n^{-\frac{4}{s}(1+\fr1p)}\g\}
\\&\leq C\log\log n\cdot n^{s-1-\fr{2}{p}}\cdot \log n\cdot n^{\fr{2}{p}-s}\\
&\leq \Gamma^4_nn^{\fr{2}{p}-s}.
\end{align*}

{\bf Step 3: $B^{1+2/p}_{p,1}$ estimate.}\, Similarly, we have
\bbal
\|(v,\beta)(t)\|_{B^{1+2/p}_{p,1}}&\leq C\int^t_0\|(v,\beta)(\tau)\|_{B^{1+2/p}_{p,1}}(1+\|U,v\|_{B^{1+2/p}_{p,1}}+\|U,\Theta\|_{W^{1,\infty}})\dd \tau
\\& \quad + C\int^t_0\|(v,\beta)(\tau)\|_{L^\infty}\|U,\Theta\|_{B^{2+2/p}_{p,1}}\dd \tau+C\int^t_0\|\Theta\|_{B^{1+2/p}_{p,1}}\dd \tau\\
&\leq C\int^t_0\|(v,\beta)(\tau)\|_{B^{1+2/p}_{p,1}}(1+\|v,\beta\|_{B^{1+2/p}_{p,1}}+\|U,\Theta\|_{B^{1+2/p}_{p,1}})\dd \tau
\\& \quad + Ct\Gamma^3_nn^{\fr{2}{p}-s}\ep_n^{1-\frac{2}{s}(2+\fr2p)}n^{2+\fr{2}{p}-s}+Ct\ep_n^{1-\frac{2}{s}(1+\fr2p)} n^{1+\fr{2}{p}-s},
\end{align*}
which implies
\bbal
\|(v,\beta)(t)\|_{B^{1+2/p}_{p,1}}&\leq C\f(\Gamma^3_n\ep_n^{-\frac{2}{s}(3+\fr2p)}n^{1+\fr{2}{p}-s}+\ep_n^{-\frac{2}{s}(2+\fr2p)}\g)\exp\f\{C\ep_n^{-\frac{4}{s}(1+\fr1p)}\g\}
\\&\leq C\f(\Gamma^3_n\cdot\log\log n\cdot n^{1+\fr{2}{p}-s}+1\g)\log n \\
&\leq \Gamma^{\fr32}_nn^{1+\fr{2}{p}-s}.
\end{align*}
From the above three steps, we deduce that \eqref{ok} holds and thus we complete the proof.
\end{proof}
\subsection{Completion of the Proof}\label{sec5}

With all the ingredients in hand, we can conclude the proof of the main theorem. Although the initial data satisfies $\|u_0,\theta_0\|_{W^{1,\infty}}\approx \varepsilon_nn^{1+\fr{2}{p}-s}$, we also show that $(u,\theta)\in C([0,T_n]; B^\infty_{p,1})$ for $T_n= n^{s-1-\fr{2}{p}}\ep_n^{-(1+\fr2s)}$.

Using Proposition \ref{pro0}, we see that
$$\|u_{0},\theta_0\|_{B^s_{p,1}}\approx \varepsilon_n.$$
Using Propositions \ref{pro2} and \ref{pro3}, we deduce that
\bbal
\|\theta(T_n)\|_{B^{s}_{p,\infty}}&\geq \|\Theta(T_n)\|_{B^{s}_{p,\infty}}-\|\beta(T_n)\|_{B^{s}_{p,1}}
\\&\geq \delta_n\f\|h(nr)\sin\f(\arctan\frac{y}{x}-\frac{\delta_n}{r} T_n\g)\g\|_{B^{s}_{p,\infty}}-\Gamma_n
\\&\geq c \delta_n n^{s-\fr{2}{p}}(n\delta_nT_n)^s-\Gamma_n\\
&\geq c\ep^{-1}_n-\Gamma_n.
\end{align*}
This
concludes the proof of Theorem \ref{th1}.

\section*{Declarations}

\noindent\textbf{Availability of data and materials}\\
 No data was used for the research described in the article.
\vspace*{1em}

\noindent\textbf{Conflict of interest}\\
The authors declare that they have no conflict of interest.
\vspace*{1em}

\noindent\textbf{Funding}\\
J. Li is supported by the Ganpo Talents Project of Jiangxi Province (No.gpyc20240069) and Natural Science Foundation of Jiangxi Province (No.20252BAC210004 and No. 20252BAC240186).

\end{document}